\documentclass{article}

\usepackage{PRIMEarxiv}

\usepackage[utf8]{inputenc} % allow utf-8 input
\usepackage[T1]{fontenc}    % use 8-bit T1 fonts
\usepackage{hyperref}       % hyperlinks
\usepackage{url}            % simple URL typesetting
\usepackage{booktabs}       % professional-quality tables
\usepackage{amsfonts}       % blackboard math symbols
\usepackage{nicefrac}       % compact symbols for 1/2, etc.
\usepackage{microtype}      % microtypography
\usepackage{lipsum}
\usepackage{graphicx}
\graphicspath{{media/}}     % organize your images and other figures under media/ folder

\usepackage{amsmath}
\usepackage{amssymb}
\usepackage{amsthm}
\usepackage{enumitem}

\newtheorem{definition}{Definition}[section]

\newtheorem{theorem}[definition]{Theorem}
\newtheorem{lemma}[definition]{Lemma}

\newtheorem{remark}[definition]{Remark}

\newcommand{\R}{\mathbb{R}}                             % ---- reals
\newcommand{\embed}{\hookrightarrow}                    % ---- embeddding

\DeclareMathOperator{\LspaceSymbol}{\mathbf{L}}

\newcommand{\Lpspace}[1][p]{\LspaceSymbol^{{#1}}}       % ---- Lebesgue space
\newcommand{\LPhispace}[1][\DefaultPhifunction]{\LspaceSymbol^{{#1}}}% ---- Orlicz space LPhi
\DeclareMathOperator{\WspaceSymbol}{\mathbf{W}}

\newcommand{\WLGspace}[1][\DefaultGfunction] { {\WspaceSymbol^1}\LspaceSymbol^{{#1}} }
\newcommand{\WLPhispace}[1][\DefaultPhifunction] { {\WspaceSymbol^1}\LspaceSymbol^{{#1}} }
\newcommand{\WLGzspace}[1][\DefaultGfunction]{ {\WspaceSymbol^1_0}\LspaceSymbol^{{#1}}}
\newcommand{\norm}[1]{\|#1\|}                           % ---- general norm

\newcommand{\LPhinorm}[3][\DefaultPhifunction]{\norm{#2}_{\LPhispace[{#1}]({#3})}}														                                          	% ---- Luxemburg norm (LPhi) zbiorze          
 
 \newcommand{\WLPhinorm}[3][\DefaultPhifunction]{\norm{#2}_{\WLPhispace[{#1}]({#3})}}

\title{Generalized version of the Lions-type lemma
}

\author{
  Magdalena Chmara\\
  Department of Technical Physics and Applied Mathematics, \\
Gda\'{n}sk University of Technology,  \\
Narutowicza 11/12, 80-233 Gda\'{n}sk, Poland \\
  \texttt{magdalena.chmara@pg.edu.pl} \\
 }

\begin{document}
\maketitle

\begin{abstract}
	In this short paper, I recall the history of dealing with the lack of compactness of a  sequence in the case of an unbounded domain and prove the vanishing Lions-type result  for a sequence of Lebesgue-measurable functions. This lemma generalizes some results for a class of Orlicz-Sobolev spaces. What matters here is the behavior of the integral, not the space. \end{abstract}

% keywords can be removed

\section{Introduction}

In 1984 P.L. Lions published his celebrated article \cite{Lio84}, in which he introduced a concentration-compactness method for solving minimization problems on unbounded domains. One of the main tools  provided by  \cite{Lio84} is  lemma I.1.  A variety of formulations of this lemma has been widely used to deal with the lack of compactness on unbounded domain for different types of equations.
In \cite[p. 102]{Cos07} we can find the following version of the Lion's Lemma:
% which gives conditions guaranteeing strong $\Lpspace[q](\R^N)$-convergence of a weakly convergent $\Lpspace(\R^N)$-sequence arising from a sequence of functions
%bounded in Sobolev Space $\mathbf{H}^1(\R^N)$.  
\begin{lemma}\label{lem:liolap}
	%	Let  $\mathbf{H}^1(\R^N)$  be a Sobolev space obtained by completion
	%	of $C_0^{\infty}(\R^N)$ in the norm\\
	% $\norm{u}:=~\left(\int_{\R^N}[|\nabla u|^2+|u|^2]\,dx\right)^{1/2}$.
	Suppose  $\{u_n\}\in\mathbf{H}^1(\R^N)$ is a bounded sequence satisfying 
	\[
	\lim_{n\to\infty}\left(\sup_{y\in\R^N}\int_{B_r(y)}|u_n|^p\right)=0
	\]
	for some $p\in[2,2^{\ast}]$ and $r>0$, where $B_r(y)$ denotes the open ball of radius $r$ centered at $y\in\R^N$. Then $u_n\to0$ strongly in $\Lpspace[q](\R^N)$ for all $2<q<2^{\ast},$ where $2^{\ast}$ is the limiting exponent in the Sobolev embedding  $\mathbf{H}^1(\R^N)\embed\Lpspace(\R^N)$.
\end{lemma}
This version of the lemma has been used to solve semilinear elliptic equation in the whole space $\R^N$, i.e. 
\[
-\Delta u+u=h(u),~ u\in\mathbf{H}^1(\R^N).
\]

 In \cite{Lew10}  and  \cite{Str08} you can find a comprehensive description of the lack of compactness in Sobolev spaces

The Lions Lemma has been generalized in some ways, for example in \cite{AlvFigGioSan14} we can find the formulation of the  lemma for isotropic Orlicz-Sobolev spaces $\WLGzspace[A](\R^N)$, i.e. spaces obtained by the completion of  $C_0^{\infty}(\R^N)$ with respect to the norm
$\WLPhinorm[A]{u}{\R^N}=\LPhinorm[A]{|\nabla u|}{\R^N}+\LPhinorm[A]{u}{\R^N}$, where 
 
\[
\LPhinorm[A]{u}{\R^N}=\inf\left\{k>0: \int_{\R^N} A\left(\frac{|u|}{k}\right)\, dt\leq 1
\right\},
\]	is a Luxemburg norm,  $A:\R\to[0,\infty)$ is an N-function (i.e.  is  convex, even, coercive and vanishes only at 0) satisfying $\Delta_2\nabla_2$ condition (i.e. there exist $K_1,K_2>0$, such that $K_1A(v)\leq A(2v)\leq K_2A(v)$ for all $v\in \R^n$).

\begin{lemma}[Theorem. 1.3 in \cite{AlvFigGioSan14}]
	\label{lem:lio:iso}
Assume that $a(t)t$ is increasing in $(0,+\infty)$ and that there exist $l,m\in(1,N)$ such that 

\begin{equation}
\label{ass:aAlm}
l\leq\frac{a(|t|)t^2}{A(t)}\leq m\quad \text{for all }t\neq0,
\end{equation}
where $A(t)=\int_0^{|t|}a(s)s\,ds,$ $l\leq m<l^{\ast}=\frac{lN}{N-l}$.
Let $\{u_n\}\subset\WLGspace[A](\R^N)$ be a bounded sequence such that there exists $R>0$ satisfying:
	\[
	\label{cond:vanishing}
	\tag{$L_1$}
\lim_{n\to\infty}\left(\sup_{y\in\R^N}\int_{B_r(y)}A(|u_n|)\right)=0.
\]
Then, for any N-function $B$ verifying $\Delta_2$-condition and satisfying 
\[\lim_{t\to 0}\frac{B(t)}{A(t)}=0\quad\text{and}\quad \lim_{t\to \infty}\frac{B(t)}{A^{\ast}(t)}=0,\] where $A^{\ast}$ is a Sobolev conjugate of $A$, w have
\[
u_n\to 0 \text{ in }\LPhispace[B](\R^N).\]
\end{lemma}
In \cite{AlvFigGioSan14} authors use lemma \ref{lem:lio:iso} to prove the existence of solutions to some isotropic  quasilinear problems.

It is worth to notice, that in the proof of the  lemma above  authors essentially use the fact that function $A$ satisfies $\Delta_2\nabla_2$ condition, which is guaranteed by condition \eqref{ass:aAlm}. Isotropic Young function satisfying globally $\Delta_2\nabla_2$ condition is bounded by some power functions  with power $1<p<\infty$. If $A$ satisfies  $\Delta_2\nabla_2$ then $\WLGspace[A]$ is a reflexive, separable Banach space (see e.g. \cite{Ada03}).

  There are also papers, where authors consider non-reflexive spaces, e.g. \cite{AlvCar22}. In this case instead of condition \eqref{cond:vanishing} authors use the assumption  \eqref{ass:lieb}  (see \cite{Lie83}) and   assume that the sequence $\left\{\int_{\R^N}A^{\ast}(|u_n|)\,dx\right\}$ is bounded. 

\begin{lemma}[Theorem. 1.3 in \cite{AlvCar22}]
	\label{lem:lio:nonref}
	Let $A, B$ be a N-functions, $A^*$ be a Sobolev conjugate of $A$ and 
		\[\lim_{t\to 0}\frac{B(t)}{A(t)}=0\quad\text{and}\quad \lim_{t\to 0}\frac{B(t)}{A^{\ast}(t)}=0.\]
If  $\{u_n\}\subset\WLGspace[A](\R^N)$ is a sequence such that 
$\left\{\int_{\R^N}A(|u_n|)\,dx\right\}$ and $\left\{\int_{\R^N}A^{\ast}(|u_n|)\,dx\right\}$ are bounded, and for each $\varepsilon>0$ we have
\begin{equation}
	\label{ass:lieb}
	\tag{$L_2$}
	\text{meas}(|u_n|>\varepsilon)\to 0 \quad as\quad n\to\infty,
\end{equation}
then

\[\int_{\R^N}B(u_n)\to 0 \quad as\quad n\to\infty.\]
\end{lemma}

In \cite{Wro22} author uses the lemma similar to lemma \eqref{lem:lio:iso}, but for sequences from anisotropic Orlicz-Sobolev spaces, to find solutions of the  anisotropic  quasilinear problem
\begin{equation} \tag{AQP} \label{eq:AQP}
	-\text{div} (\nabla \Phi(\nabla u))+V(x)N'(u)=f(u),\quad \text{ where }
	u\in\WLPhispace(\R^n),
\end{equation} 
where $\Phi$ is an anisotropic n-dimensional N-function (see more in \cite{BarCia17}), satisfying $\Delta_2\nabla_2$ condition.

In \cite{SilCarAlbBah21} authors prove Lion's type lemma for reflexive fractional Orlicz-Sobolev spaces, while in  \cite{BahOuElf22} authors prove it for non-reflexive fractional Orlicz-Sobolev spaces.

\section{Main Theorem}
In this paper we generalize the Lions-type lemmas \ref{lem:liolap}, \ref{lem:lio:iso}, \ref{lem:lio:nonref}, we mentioned in the introduction.   The only assumption on functions is that they are Lebesgue measurable finite and  vanish only at zero.  It is worth to notice, that they can have growth which is not bounded by polynomials, so it will be possible to use this lemma also in non-reflexive spaces.  In the proof of the following lemma we use some techniques from \cite{SilCarAlbBah21}.

\begin{theorem}\label{thm:lions}
	Assume that $\Phi_1,\Phi_2, \Psi:\R^n\to[0,\infty)$ are Lebesgue-measurable  functions vanishing only in zero, satisfying
	\begin{equation}\label{eq:phi/psi=0inzero}\tag{$\Psi_1$}
	 	\lim_{|v|\to 0} \frac{\Psi(v)}{\Phi_1(v)}=0,
	\end{equation}
	\begin{equation}\label{eq:phi/psi=0ininfty}\tag{$\Psi_2$}
	\lim_{|v|\to \infty} \frac{\Psi(v)}{\Phi_2(v)}=0,
\end{equation}

 Let $\{u_k\}$ be a sequence of Lebesgue-measurable functions  $u_k:\R^N\to\R^n$ such that $\left\{\int_{\R^N}\Phi_1(u_k)\right\},$   $\left\{\int_{\R^N}\Phi_2(u_k)\right\}$ are bounded and
\begin{equation}\label{eq:lionslimit}
\lim_{k\to\infty}\left[\sup_{y\in\R^N}\int_{B_r(y)}\Phi_1(u_k)\right]=0
\end{equation}
for some $r>0$.
Then \[
\lim_{k\to\infty}\int_{\R^N}\Psi(u_k)=0
\]
	\end{theorem}
\begin{proof}
	We let  $|A|$ denote the Lebesgue measure of subset $A$.
 Let $\{u_k\}$ be a sequence of Lebesgue-measurable functions   such that $\left\{\int_{\R^N}\Phi_1(u_k)\right\},$   $\left\{\int_{\R^N}\Phi_2(u_k)\right\}$ are bounded.
 
  Define \[M_1=\sup_{k}\int_{\R^N}\Phi_1(u_k)\quad M_2=\sup_{k}\int_{\R^N}\Phi_2(u_k).\]
  
Note that $M_1,~M_2<\infty$.
 Fix $\varepsilon>0$.
  From \eqref{eq:phi/psi=0inzero}, there exists $\delta>0$, such that 
\begin{equation}\label{eq:psiphieps}
  \frac{\Psi(v)}{\Phi_1(v)}\leq \frac{\varepsilon}{3M_1}
\end{equation}
  for all $|v|\leq \delta$.
  
   Similarly from \eqref{eq:phi/psi=0ininfty}, there exists $T>0$, such that 
\begin{equation}\label{eq:psiphineps}
  \frac{\Psi(v)}{\Phi_2(v)}\leq \frac{\varepsilon}{3M_2}
 \end{equation}
  for all $|v|\geq T$.
 Let us denote: 
 \[
 A_k=\left\{x\in\R^N\colon|u_k(x)|\leq \delta\right\},\quad  B_k=\left\{x\in\R^N\colon\delta<|u_k(x)|<T\right\},\quad  C_k=\left\{x\in\R^N\colon|u_k(x)|\geq T\right\}.  \]
 Then \[
 \int_{\R^N}\Psi(u_k)=\int_{A_k}\Psi(u_k)+\int_{B_k}\Psi(u_k)+\int_{C_k}\Psi(u_k).
 \]
By \eqref{eq:psiphieps} we obtain
  \[
\int_{A_k}\Psi(u_k)\leq \frac{\varepsilon}{3M_1}\int_{R^N}\Phi_1(u_k)\leq \frac{\varepsilon}{3}
\]
  and by  \eqref{eq:psiphineps}
  \[
 \int_{C_k}\Psi(u_k)\leq \frac{\varepsilon}{3M_2}\int_{R^N}\Phi_2(u_k)\leq \frac{\varepsilon}{3}.
 \]
 We need to show that \[\int_{B_k}\Psi(u_k)\leq\frac{\varepsilon}{3}.\] 
 We will first show that $|B_k|\to0$ as $k\to\infty$.

 Assume, by contradiction, that (up to subsequence) \[|B_k|\to L>0.\]
 Then, for some subsequence $\{u_k\},$ there exist $y_0\in\R^N$ and $\sigma>0$ such that 
 \begin{equation}
 	\label{eq:|BkBr|>>0}
 	|B_k\cap B_r(y_0)|\geq \sigma>0.
 \end{equation}
Let
\[
C_{\Psi}=\max_{\delta\leq|v|\leq T}\Psi(v),\quad c_{\Phi}=\min_{\delta\leq|v|\leq T}\Phi_1(v),\quad  C_{\Phi}=\max_{\delta\leq|v|\leq T}\Phi_1(v).
\]
We observe that   \[
 \int_{B_r(y_0)}\Phi_1(u_k)\geq \int_{B_r(y_0)\cap B_k}\Phi_1(u_k)\geq c_{\Phi}|B_k\cap B_r(y_0)|.
 \]
 Hence and by assumption \eqref{eq:lionslimit} we have that 
 \[
 |B_k\cap B_r(y_0)|\to 0 \quad\text{as }k\to\infty
 \]
which contradicts with \eqref{eq:|BkBr|>>0}.

Since $|B_k|\to0$ as $k\to\infty$, we have that there exists $k_0$ such that for all $k\geq k_0$

% \[|B_k|<\frac{\underline{C}_{\Phi}}{C_{\Psi}\overline{C}_{\Phi}}\frac{\varepsilon}{3}.\]

\[|B_k|<c_{\Phi}\left(3C_{\Phi}C_{\Psi}\right)^{-1}\varepsilon.\]

Then 
\[
|B_k|\leq \left( c_{\Phi}\right)^{-1}\int_{B_k}\Phi(u_k)\leq C_{\Phi}\left(c_{\Phi}\right)^{-1}|B_k|
\]
and
\[
\int_{B_k}\Psi(u_k)\leq C_{\Psi}\left(c_{\Phi}\right)^{-1}\int_{B_k}\Phi(u_k)\leq C_{\Psi}C_{\Phi}\left(c_{\Phi}\right)^{-1}|B_k|<\frac{\varepsilon}{3}.
\]
\end{proof}

\begin{remark}
Note that what matters in this theorem (just as in the concentration-compactness lemma of Lions 
in \cite{Lie83}) is the behavior of the integral, not the space. 
\end{remark}

%Bibliography
\bibliographystyle{unsrt}  
\bibliography{references}

\end{document}